\documentclass[11pt,openbib]{amsart}

\usepackage{amsmath}
\usepackage{amssymb}
\usepackage{amsthm}
\usepackage{amscd}

\newcommand{\ri}{\mathfrak{o}}
\newcommand{\mi}{\mathfrak{p}}

\newcommand{\Z}{\mathbf{Z}}
\newcommand{\C}{\mathbf{C}}

\newcommand{\e}{\sqrt{\epsilon}}
\newcommand{\p}{\varpi}

\newcommand{\U}{\mathrm{U}}

\newcommand{\Sch}[1]{\mathcal{C}_c^\infty({#1})}
\newcommand{\supp}{\operatorname{supp}}
\newcommand{\one}{\mathbf{1}}
\newcommand{\ru}{\mathbf{u}}

\numberwithin{equation}{section}
\newtheorem{thm}[equation]{Theorem}
\newtheorem{lem}[equation]{Lemma}
\newtheorem{prop}[equation]{Proposition}

\theoremstyle{definition}
\newtheorem{defn}[equation]{Definition}

\theoremstyle{remark}
\newtheorem{rem}[equation]{Remark}

\theoremstyle{definition}

\theoremstyle{remark}

\def\Section#1{\section{#1}\setcounter{equation}{0}}

\begin{document}

\title{On $L$-factors attached to generic representations of unramified $\U(2,1)$}

\author{Michitaka Miyauchi}
\address{Faculty of Liberal Arts and Sciences\\
Osaka Prefecture University\\
1-1 Gakuen-cho, Nakaku, Sakai, Osaka 599-8531, JAPAN}
\email{michitaka.miyauchi@gmail.com}
%\thanks{}
%\date{}

%\dedicatory{}

%    Abstract is required.
\begin{abstract}
Let $G$ be the unramified unitary group in three variables
defined over a $p$-adic field with $p \neq 2$.
In this paper,
we establish a theory of newforms for 
the Rankin-Selberg integral  for $G$
introduced by Gelbart and Piatetski-Shapiro.
We describe $L$ and $\varepsilon$-factors 
defined through zeta integrals in terms of newforms.
We show that 
zeta integrals of newforms for generic representations 
attain $L$-factors.
As a corollary,
we get an explicit formula for $\varepsilon$-factors
of generic representations.
\end{abstract}

\maketitle
\pagestyle{myheadings}
\markboth{}{}

\section{Introduction}
This paper is the sequel to the author's works
\cite{M3}, \cite{M2} and \cite{M}
on newforms for unramified $\mathrm{U}(2,1)$.
First of all,
we review the theory of newforms for 
$\mathrm{GL}(2)$ by Casselman and Deligne.
Let $F$ be a non-archimedean local field of characteristic zero
with ring of integers $\ri_F$
and its maximal ideal $\mi_F$.
For each non-negative integer $n$,
we define an open compact subgroup 
$\Gamma_0(\mi_F^n) $
of $\mathrm{GL}_2(F)$
by
\[
\Gamma_0(\mi_F^n) 
= 
\left(
\begin{array}{cc}
\ri_F & \ri_F \\
\mi_F^n & 1+\mi_F^n
\end{array}
\right)^\times.
\]
For an  irreducible generic
representation $(\pi, V)$ of $\mathrm{GL}_2(F)$,
we denote by $V(n)$
the $\Gamma_0(\mi_F^n)$-fixed subspace of $V$,
that is,
\[
V(n) = \{v \in V\, |\, \pi(k)v = v,\ k \in \Gamma_0(\mi_F^n)\}.
\]
Let $U$ denote the unipotent radical
of the upper-triangular Borel subgroup of $\mathrm{GL}_2(F)$.
We regard
a non-trivial additive character $\psi_F$ of $F$
with conductor $\ri_F$
as a character of $U$ in the usual way,
and denote by $\mathcal{W}(\pi, \psi_F)$
the Whittaker model of $\pi$ with respected to $\psi_F$.
Then 
the following 
theorem holds:
%%%%%%%%%%%%%%%%%%%%%
\begin{thm}[\cite{Casselman}]\label{thm:gl2}
Let $(\pi, V)$ be an irreducible generic 
representation of $\mathrm{GL}_2(F)$.

(i) There exists a non-negative integer $n$ such that
$V(n) \neq \{0\}$.

(ii) Put $c(\pi) = \min \{n\geq 0\, |\, V(n) \neq \{0\}\}$.
Then the space $V(c(\pi))$ is one-dimensional.

(iii) For any $n \geq c(\pi)$,
we have
$\dim V(n) = n-c(\pi)+1$.

(iv)
If $v$ is a non-zero element in $V(c(\pi))$,
then the corresponding Whittaker function 
$W_v$ in $\mathcal{W}(\pi, \psi_F)$ 
satisfies $W_v(e) \neq 0$,
where $e$ denotes the identity element in $\mathrm{GL}_2(F)$.
\end{thm}
We call
the integer $c(\pi)$ {\it the conductor of} $\pi$
and $V(c(\pi))$ {\it the space of newforms for} $\pi$.
Newforms and conductors relate to 
$L$ and $\varepsilon$-factors as follows:
%%%
%%%
\begin{thm}[\cite{Casselman}, \cite{Deligne}]\label{thm:D}
Let $\pi$ be 
an irreducible generic 
representation of $\mathrm{GL}_2(F)$.

(i)
Suppose that
$W$ is the newform in the  Whittaker model of $\pi$.
Then the corresponding Jacquet-Langlands's
zeta integral $Z(s, W)$
attains the $L$-factor of $\pi$.

(ii)
The $\varepsilon$-factor $\varepsilon(s, \pi, \psi_F)$ of $\pi$
 is a constant multiple of $q_F^{-c(\pi)s}$,
 where $q_F$ stands for the cardinality of the residue field of $F$.
\end{thm}

Similar results were obtained by Jacquet, Piatetski-Shapiro and Shalika \cite{JPSS}
and Reeder \cite{Reeder}
for $\mathrm{GL}(n)$.
Recently, Roberts and Schmidt \cite{RS}
established a theory of newforms 
for the irreducible representations of $\mathrm{GSp}(4)$
with trivial central characters.
Our main concern is to establish a newform theory for 
unramified $\mathrm{U}(2,1)$.

We review
results in \cite{M3}, \cite{M2} and \cite{M}
comparing Theorems~\ref{thm:gl2} and \ref{thm:D}.
Let $\mathrm{U}(2,1)$ denote 
the unitary group in three variables associated to 
the unramified quadratic extension $E/F$.
We assume that the residual characteristic of $F$
is odd.
In \cite{M},
the author introduced a family of open compact subgroups
of $\mathrm{U}(2,1)$,
and
defined the notion of
conductors and newforms for generic representations.
He proved  an analog of Theorem~\ref{thm:gl2}
(i) and (ii) for all the generic representations,
and that of (iii) and (iv)
for the generic supercuspidal representations.
For $\mathrm{U}(2,1)$,
we consider $L$ and $\varepsilon$-factors
defined through
the Rankin-Selberg integral
introduced by Gelbart, Piatetski-Shapiro \cite{GPS}
and Baruch \cite{Baruch}.
In \cite{M2},
the author showed a theorem analogous to
Theorem~\ref{thm:D} (ii)
assuming Conjecture 3.1 in \cite{M2} on 
$L$-factors,
which is an analog of Theorem~\ref{thm:D} (i).
In {\it loc. cit.},
he also proved that 
his conjecture holds for the generic supercuspidal
representations.
To show the validity of his conjecture for the 
generic representations,
he determined conductors of the generic non-supercuspidal 
representations,
and gave an explicit realization of 
those newforms in \cite{M3}.
In {\it loc. cit.},
he also proved an analog of Theorem~\ref{thm:gl2}
(iii) and (iv)
for the generic non-supercuspidal representations.
Now we are ready to show 
that Conjecture 3.1 in \cite{M2} 
holds for all the generic representations
of $\mathrm{U}(2,1)$,
that is, 
zeta integrals of newforms attain 
$L$-factors.

We explain our method.
Unlike the cases of $\mathrm{GL}(n)$
and $\mathrm{GSp}(4)$,
Gelbart and Piatetski-Shapiro's zeta integral 
involves a section which has the form $f(s, h, \Phi)$, 
where $h$ is an element in $\mathrm{U}(1,1)$ and $\Phi$
is a Schwartz function on $F^2$.
Thus, the usual investigation on 
Whittaker functions is not enough 
to determine the $L$-factor, which is defined as
 the greatest common divisor of 
zeta integrals,
and we can not use 
any explicit formula of 
$L$-factors for $\mathrm{U}(2,1)$.
However it is easy to determine 
the $L$-factors for $\mathrm{U}(2,1)$
up to a multiple of $L_E(s, \one)$
(Proposition~\ref{prop:L_es}).
Here $L_E(s, \one)$ stands for the 
Hecke-Tate factor of the trivial representation $\one$ of $E^\times$,
and the section $f(s, h, \Phi)$ yields $L_E(s, \one)$.
We will compare zeta integral of newforms
with our rough estimation of $L$-factors,
and show that 
the difference is at most $L_E(s, \one)$ (Lemma~\ref{lem:ZL}).
Hence we can use the same trick in \cite{M2}.
If the difference is $L_E(s, \one)$,
then it contradicts the fact that
the $\varepsilon$-factor is monomial (see the proof of Theorem~\ref{thm:main}).
So we conclude that zeta integrals of newforms
attain $L$-factors.

The main body of this article is 
the proof of Lemma~\ref{lem:ZL}.
For representations of conductor zero,
we can use Casselman-Shalika's formula for
the spherical Whittaker functions in \cite{CS}.
To 
compute zeta integrals of newforms in positive conductor case,
we follow the method by Roberts and Schmidt
for $\mathrm{GSp}(4)$ in  \cite{RS}.
They utilized
Hecke operators
acting on the space of newforms,
and obtained a formula for zeta integrals 
in terms of  Hecke eigenvalues.
There are two problems to apply their method to $\mathrm{U}(2,1)$.
Firstly,
they assumed that representations of $\mathrm{GSp}(4)$ 
have trivial central characters.
This assumption is essential in their computation 
of Hecke operators.
Secondly,
for an irreducible generic representation $\pi$
of $\mathrm{U}(2,1)$ whose conductor is positive,
it will turns out that
the degree of  the $L$-factor of $\pi$
is at most 4 with respect to $q_F^{-s}$ (see 
Proposition~\ref{prop:iu} for example).
Therefore we need two Hecke eigenvalues 
to describe zeta integrals of newforms.
But, in the usual way,
we have only one good Hecke operator 
which is represented by the element
$\mathrm{diag}(\p, 1, \p^{-1})$,
where $\p$ is a uniformizer of $F$.
We explain how to overcome these two problems.
Let $V$ denote the space of $\pi$,
$V(n)$ its subspace of vectors fixed by the level $n$ subgroup,
and $N_\pi$ the conductor of $\pi$.
We consider the following two operators:
\begin{enumerate}
\item
The Hecke operator $T$ on $V(N_\pi+1)$
which is represented by the element
$\mathrm{diag}(\p, 1, \p^{-1})$;
\item
The composite map of the level raising operator $\theta':
V(N_\pi) \rightarrow V(N_\pi+1)$
and the level lowering one
$\delta:
V(N_\pi+1) \rightarrow V(N_\pi)$.
\end{enumerate}
In \cite{M3},
we have seen that 
both $V(N_\pi)$ and $V(N_\pi+1)$
are one-dimensional,
and hence
the operators $T$ and $\delta \circ \theta'$
have eigenvalues $\nu$ and $\lambda$.
Since the central character of $\pi$ is trivial on 
the level $N_\pi$ subgroup,
we can apply the method by Roberts and Schmidt 
to compute the Hecke operator $T$ on $V(N_\pi+1)$,
and get a formula of zeta integrals of 
newforms in terms of $\nu$ and $\lambda$
(Theorem~\ref{thm:zeta1}).

We summarize the contents of this paper.
In section~\ref{sec:GPS}, 
we fix the notation for representations of unramified $\mathrm{U}(2,1)$,
and recall the theory of 
Rankin-Selberg integrals introduced by Gelbart,
Piatetski-Shapiro and Baruch.
In section~\ref{sec:RS},
we recall the notion of newforms for $\mathrm{U}(2,1)$,
and prove our main Theorems~\ref{thm:main}
and \ref{thm:main2},
assuming Lemma~\ref{lem:ZL}.
In section~\ref{sec:L},
we roughly estimate $L$-factors 
according to the classification of the representations of 
$\mathrm{U}(2,1)$.
In section~\ref{section:zeta},
we give a formula for zeta integrals of newforms
in terms of two eigenvalues $\nu$ and $\lambda$.
The proof of Lemma~\ref{lem:ZL}
is finished in section~\ref{section:pf}.
In section~\ref{sec:example},
we give an example of an explicit computation 
of $L$-factors,
for some non-supercuspidal representations.

A further direction of this research
is to compare $L$ and $\varepsilon$-factors
defined by Gelbart and Piatetski-Shapiro's integral with 
those of $L$-parameters.
It is also an interesting problem 
to generalize our result to other $p$-adic groups,
for example,
ramified $\mathrm{U}(2,1)$
and unitary groups in odd variables.

\section{Gelbart and Piatetski-Shapiro's integral}\label{sec:GPS}
In subsection~\ref{subsec:notation},
we fix our notation for the 
unramified group $\mathrm{U}(2,1)$
that we use throughout this paper.
In subsection~\ref{subsec:zeta},
we recall from \cite{Baruch} the theory of 
zeta integrals for $\mathrm{U}(2,1)$
which is 
introduced by Gelbart and Piatetski-Shapiro
in \cite{GPS}.
We also recall the definition of $L$ and $\varepsilon$-factors
attached to generic representations of $\mathrm{U}(2,1)$
in subsections~\ref{subsec:L} and \ref{subsec:epsilon}
respectively.

%%%%%
%%%%%
%%%%%
\subsection{Notations}\label{subsec:notation}
Let $F$ be a non-archimedean local field of characteristic zero,
$\ri_F$ its ring of integers,
$\mi_F$ the maximal ideal in $\ri_F$,
and $\p = \p_F$ a uniformizer of $F$.
We denote by $|\cdot|_F$ the absolute value of $F$
normalized so that $|\p_F|_F = q^{-1}$,
where $q = q_F$ is the cardinality of 
the residue field $\ri_F/\mi_F$.
We use the analogous notation 
for any non-archimedean local fields.
Throughout this paper,
we  assume that
the residual characteristic of
$F$ is different from two.

Let $E = F[\e]$ be the unramified quadratic  extension over $F$,
where $\epsilon$ is a non-square element in $\ri_F^\times$.
Then $\p = \p_F$ is a common uniformizer of $E$ and $F$.
Because the cardinality of the residue field of $E$
is equal to $q^2$,
we denote by $|\cdot|_E$
the absolute value of $E$ normalized so that 
$|\p|_E = q^{-2}$.
We realize the unramified 
unitary group in three variables defined over $F$
as
$G  = 
\{ g \in \mathrm{GL}_3(E)\ |\ 
{}^t \overline{g} Jg = J \}$,
where ${}^-$ is the non-trivial element in $\mathrm{Gal}(E/F)$
and
\begin{eqnarray*}
J = 
\left(
\begin{array}{ccc}
0 & 0 &1\\
0 & 1 & 0\\
1 & 0 & 0
\end{array}
\right).
\end{eqnarray*}
We denote by $e$ the identity element of $G$.

Let $B$ be the Borel subgroup of $G$
 consisting of the upper triangular elements
in $G$,
$T$ its
diagonal subgroup,
and $U$ the
unipotent radical of $B$.
We write $\hat{U}$ for the  opposite of $U$.
Then we have
\begin{align*}
U& = \left\{
u(x, y) =
\left(
\begin{array}{ccc}
1 & x & y\e-x\overline{x}/2\\
0 & 1 & -\overline{x}\\
0 & 0 & 1
\end{array}
\right)\, 
\Bigg|\,
x \in E, y \in F
\right\}\\
& = \left\{
\ru(x, y) =
\left(
\begin{array}{ccc}
1 & x & y\\
0 & 1 & -\overline{x}\\
0 & 0 & 1
\end{array}
\right)\, 
\Bigg|\,
x, y \in E,\,
y+\overline{y}+x\overline{x}=0
\right\}
\end{align*}
and
\begin{align*}
\hat{U}&  = \left\{ \hat{u}(x,y) =
{}^t\! u(x,y)
|\,
x\in E, y \in F
\right\}\\
& = \left\{ \hat{\ru}(x,y) =
{}^t\! \ru(x, y)
|\,
x, y \in E,\, 
y+\overline{y}+x\overline{x}=0
\right\},
\end{align*}
where ${}^t$ denotes the transposition of matrices.
In most part of this paper,
we write $u(x, y)$ for elements in $U$.
The notion $\ru(x, y)$ will appear only in 
the proofs of Lemmas~\ref{lem:n+1}
and \ref{lem:RU21}.
We identify the subgroup
\begin{eqnarray*}
H = \left\{
\left(
\begin{array}{ccc}
a & 0 & b\\
0 & 1 & 0\\
c & 0 & d
\end{array}
\right) \in G
\right\}
\end{eqnarray*}
of $G$ with
$\U(1,1)$.
We  set $B_H = B\cap H$, $U_H = U\cap H$
and $T_H = T\cap H$.
Then $B_H$ is the upper triangular Borel subgroup of $H$
with Levi decomposition $B_H = T_H U_H$.
There exists an isomorphism between $E^\times$ and $T_H$
which is given by
\begin{eqnarray*}
t:\, E^\times \simeq T_H;\, a \mapsto
t(a) = \left(
\begin{array}{ccc}
a & 0 & 0\\
0 & 1 & 0\\
0 & 0 & \overline{a}^{-1}
\end{array}
\right).
\end{eqnarray*}

A non-trivial additive character $\psi_E$ of $E$
defines the following character of $U$,
which is also denoted by $\psi_E$:
\[
\psi_E(u(x, y)) = \psi_E(x),\ \mathrm{for}\
u(x, y) \in U.
\]
We say that
a smooth representation $\pi$ of $G$ is {\it generic} if
$\mathrm{Hom}_U(\pi, \psi_E) \neq \{0\}$.
Let $(\pi, V)$ be an irreducible generic representation
of $G$.
Then there exists a unique embedding of
$\pi$ into $\mathrm{Ind}_U^G \psi_E$
up to scalars.
The image $\mathcal{W}(\pi, \psi_E)$ of $\pi$
in $\mathrm{Ind}_U^G \psi_E$
 is called {\it the Whittaker model of} $\pi$.
Given a non-zero element $l$ in $\mathrm{Hom}_U(\pi, \psi_E)$,
we define {\it the Whittaker function} $W_v$ in $\mathcal{W}(\pi, \psi_E)$
associated to $v \in V$ by
\[
W_v(g) = l(\pi(g)v),\ g \in G.
\]

We identify
the center $Z$ of $G$
with the norm-one subgroup $E^{1}$ of $E^\times$,
and define open compact subgroups
of $Z$ by
\[
Z_0 = Z,\ Z_n = Z\cap (1+\mi_E^n),\ \mathrm{for}\ n \geq 1.
\]
For an irreducible admissible representation $\pi$ of $G$,
we define {\it the conductor} $n_\pi$ {\it of} 
the central character $\omega_\pi$ of $\pi$
by
\begin{eqnarray*}
n_\pi = \mathrm{min}\{n \geq 0\, |\, \omega_\pi|_{Z_n} = 1\}.
\end{eqnarray*}

%%%%%
%%%%%
%%%%%
\subsection{Zeta integrals}\label{subsec:zeta}
Let $\Sch{F^2}$ denote the space of locally constant, compactly supported functions on $F^2$.
For $\Phi \in \Sch{F^2}$ and $g \in \mathrm{GL}_2(F)$,
we define a function $z(s, g, \Phi)$ on $\C$ by
\begin{eqnarray*}
z(s, g, \Phi)
=
\int_{F^\times}\Phi((0, r)g) |r|_E^s d^\times r,\ s \in \C.
\end{eqnarray*}
Here
we normalize the Haar measure $d^\times r$ on $F^\times$ so that
the volume of $\ri_F^\times$ is one.

For $a \in E^\times$,
we set
$t(a) = \left(
\begin{array}{cc}
a & 0\\
0 & \overline{a}^{-1}
\end{array}
\right)$ and 
$d(a) = \left(
\begin{array}{cc}
a & 0\\
0 & 1
\end{array}
\right)$.
Since $\mathrm{SU}(1,1)$ is isomorphic to 
$\mathrm{SL}_2(F)$,
we can write any element $h$ in $H = \U(1,1)$
as
\begin{eqnarray}\label{eq:h}
h = t(b) d(\e) h_1 d(\e^{-1}),\
\end{eqnarray}
where $b \in E^\times$ and $h_1 \in \mathrm{SL}_2(F)$.
For $h \in H$ and $\Phi \in \Sch{F^2}$,
using the decomposition of $h$ in (\ref{eq:h}),
we define a function $f(s, h, \Phi)$ on $\C$ by
\begin{eqnarray*}
f(s, h, \Phi)
=|b|_E^s z(s, h_1, \Phi),\ s \in \C.
\end{eqnarray*}
By \cite{Baruch} Lemma 2.5,
the definition of 
$f(s, h, \Phi)$ is independent of the choices of 
$b \in E^\times$ and $h_1 \in \mathrm{SL}_2(F)$
in (\ref{eq:h}).

Let $\pi$ be an irreducible generic representation of $G$.
For $W \in \mathcal{W}(\pi, \psi_E)$
and $\Phi \in \Sch{F^2}$,
we define the zeta integral
$Z(s, W, \Phi)$ by
\begin{eqnarray*}
Z(s, W, \Phi)
=
\int_{U_H\backslash H}W(h)f(s, h, \Phi) dh,
\end{eqnarray*}
where $dh$ is the Haar measure on $U_H\backslash H$
normalized so that 
the volume of $U_H\backslash U_H (H\cap \mathrm{GL}_2(\ri_F))$ is one.
By \cite{Baruch} Proposition 3.4,
$Z(s, W, \Phi)$ absolutely
converges to a  function in $\C(q^{-2s})$ when  $\mathrm{Re}(s)$ is sufficiently large.

%%%%%%%%%
%%%%%%%%%
\subsection{$L$-factors}\label{subsec:L}
The $L$-factor of an irreducible generic representation 
$\pi$ of $G$ is defined as follows.
Let $I_{\pi}$ be
the subspace of $\C(q^{-2s})$
spanned by $Z(s, W, \Phi)$ where $\Phi \in \Sch{F^2}$,
$W \in \mathcal{W}(\pi, \psi_E)$ and $\psi_E$
runs over all of the non-trivial additive characters of $E$.
By \cite{Baruch} p. 331, 
$I_{\pi}$ is a fractional ideal of $\C[q^{-2s}, q^{2s}]$
which contains $\C$.
Thus, there exists a polynomial $P(X)$ in $\C[X]$ such that
$P(0) = 1$ and $1/P(q^{-2s})$ generates $I_{\pi}$
as $\C[q^{-2s}, q^{2s}]$-module.
We define the $L$-factor $L(s, \pi)$ of $\pi$ by
\begin{eqnarray*}
L(s, \pi) = \frac{1}{P(q^{-2s})}.
\end{eqnarray*}

%%%%%
%%%%%
%%%%%
\subsection{$\varepsilon$-factors}\label{subsec:epsilon}
Let $\psi_F$ be a non-trivial additive character of $F$
with conductor
$\mi_F^{c(\psi_F)}$.
We normalize the Haar measure on $F^2$ so that
the volume of $\ri_F \oplus \ri_F$ equals to $q^{c(\psi_F)}$.
For each $\Phi \in \Sch{F^2}$,
we define its Fourier transform $\hat{\Phi}$
by
\begin{eqnarray*}
\hat{\Phi}(x, y)
= \int_{F^2} \Phi(u, v) \psi_F(yu-xv) dudv.
\end{eqnarray*}
Then we have $\hat{\hat{\Phi}} = \Phi$ for all 
$\Phi \in \Sch{F^2}$.
Due to \cite{Baruch} Corollary 4.8,
there exists a rational function $\gamma(s, \pi, \psi_F, \psi_E)$
in $q^{-2s}$ which satisfies
\begin{eqnarray*}\label{eq:fe}
\gamma(s, \pi, \psi_F, \psi_E)Z(s, W, \Phi)
=Z(1-s, W, \hat{\Phi}).
\end{eqnarray*}

We define
the $\varepsilon$-factor $\varepsilon(s, \pi, \psi_F, \psi_E)$ 
of $\pi$ 
by
\begin{eqnarray*}
\varepsilon(s, \pi, \psi_F,  \psi_E)  = 
\gamma(s, \pi, \psi_F, \psi_E) \frac{L(s, \pi)}{L(1-s, \widetilde{\pi})},
\end{eqnarray*}
where $\widetilde{\pi}$ denotes the representation contragradient to
$\pi$.
By \cite{M2} Proposition 2.12,
we have $L(s, \widetilde{\pi}) = L(s, \pi)$,
and hence
\begin{eqnarray}\label{eq:epsilon}
\varepsilon(s, \pi, \psi_F,  \psi_E)  = 
\gamma(s, \pi, \psi_F, \psi_E) \frac{L(s, \pi)}{L(1-s, {\pi})}.
\end{eqnarray}

For $\varepsilon$-factors,
the following holds:
%%%%%%%%%%%%
\begin{prop}[\cite{M2} Proposition 2.14]\label{prop:mono}
The $\varepsilon$-factor $\varepsilon(s, \pi, \psi_F, \psi_E)$ 
is a monomial in $q^{-2s}$
which has the form
\[
\varepsilon(s, \pi, \psi_F, \psi_E)
= \pm q^{-2n(s-1/2)},
\]
for some $n \in \Z$.
\end{prop}

\section{Newforms and $L$-factors}\label{sec:RS}
In subsection~\ref{subsec:newform},
we recall from \cite{M}
the notion of conductors and newforms for generic representations $\pi$ of $G$.
In subsection~\ref{subsec:main},
we prove our two main theorems
assuming Lemma~\ref{lem:ZL}.
We show that
a newform for $\pi$ attains
the $L$-factor of $\pi$
through Gelbart and Piatetski-Shapiro's integral (Theorem~\ref{thm:main}).
As a corollary,
we obtain the coincidence of 
 the conductor of $\pi$
 and
the exponent of $q^{-2s}$ of the $\varepsilon$-factor
of $\pi$ (Theorem~\ref{thm:main2}).
Lemma~\ref{lem:ZL}
will be proved in section~\ref{section:pf}.

%%%%%%%%%%%%%%%
%%%%%%%%%%%%%%%
\subsection{Newforms}\label{subsec:newform}
For a non-negative integer $n$,
we define an open compact subgroup $K_n$ of $G$
by
\begin{eqnarray*}
K_n
=
\left(
\begin{array}{ccc}
\ri_E & \ri_E & \mi_E^{-n}\\
\mi_E^n & 1+\mi_E^n & \ri_E\\
\mi_E^n & \mi_E^n & \ri_E
\end{array}
\right)
\cap G.
\end{eqnarray*}
For
an irreducible generic representation $(\pi, V)$
of $G$,
we set
\[
V(n) =\{ v \in V\, |\, \pi(k)v = v,\ k \in K_n\},\
n \geq 0.
\]
Then, by
\cite{M} Theorem 2.8,
there exists a non-negative integer $n$ such that
$V(n)$ is not zero.
%%%%%
\begin{defn}
Let $(\pi, V)$ be an irreducible generic representation 
of $G$.
We call 
the integer $N_\pi = \mathrm{min}\{n\geq 0\, |\, V(n) \neq \{0\}\}$ {\it the conductor of $\pi$}
and elements
in $V(N_\pi)$ {\it newforms for $\pi$}.
\end{defn}

It follows from \cite{M}
Theorem 5.6 
that
the space $V(N_\pi)$ 
is one-dimensional.
We shall relate newforms
with Gelbart and Piatetski-Shapiro's integral.
For $W\in \mathcal{W}(\pi, \psi_E)$,
we define the zeta integral
$Z(s, W)$ of $W$ 
by
\[
Z(s, W)
= 
\int_{E^\times} W(t(a)) |a|_E^{s-1} d^\times a.
\]
Here we normalize the Haar measure $d^\times a$
on $E^\times$ so that the volume of $\ri_E^\times$
is one.
By the proof of \cite{Baruch} Proposition 3.4,
the integral $Z(s, W)$ absolutely 
converges to a  function in $\C(q^{-2s})$
when $\mathrm{Re}(s)$ is enough large.

For each integer $n$,
let $\Phi_n$ be
the characteristic function of 
$\mi_F^n \oplus \ri_F$.
We denote by $L_E(s, \chi)$
the $L$-factor of a quasi-character $\chi$ of $E^\times$,
that is,
\[
L_E(s, \chi)
= 
\left\{
\begin{array}{cl}
\displaystyle \frac{1}{1-\chi(\p)q^{-2s}}, & \mathrm{if}\ \chi\ \mathrm{is\ unramified};\\
1, & \mathrm{if}\ \chi\ \mathrm{is\ ramified}.\\
\end{array}
\right.
\]
We write $\one$ for the trivial character of $E^\times$.
Then the following holds:
%%%%%%%%%
\begin{prop}[\cite{M2} Proposition 2.4]\label{prop:zeta0}
Suppose that a function $W$ 
in $\mathcal{W}(\pi, \psi_E)$
is fixed by $K_n$.
Then we have
\[
Z(s, W, \Phi_n)
= Z(s, W)L_E(s, \one).
\]
\end{prop}

If the conductor of $\psi_E$ is $\ri_E$,
then it follows from \cite{M3} Proposition 5.1
that
any non-zero element $v\in V(N_\pi)$
satisfies
$W_v(e) \neq 0$.
Hence there exists a newform $v$ for $\pi$
such that $W_v(e) = 1$.
We state the key lemma 
which will be proved
in section \ref{section:pf}.
%%%%%%%%%
\begin{lem}\label{lem:ZL}
Suppose that the conductor of $\psi_E$ is $\ri_E$.
Let $W$ be the Whittaker function 
associated to a newform
for $\pi$
such that $W(e) = 1$.
Then we have
\[
Z(s, W, \Phi_{N_\pi})
= L(s, \pi)\ \mbox{or}\ L(s, \pi)/L_E(s, \one).
\]
\end{lem}

%%%%%
%%%%%
%%%%%
\subsection{Main theorems}\label{subsec:main}
We shall prove our main theorems.
On $L$-factors,
we obtain the following:
%%%%%%%%
\begin{thm}\label{thm:main}
We fix an additive character $\psi_E$ of 
$E$ 
with conductor $\ri_E$.
Let $\pi$ be an irreducible generic representation of $G$,
and $v$ the newform for $\pi$ such that
$W_v(e)= 1$.
Then we have
\[
Z(s, W_v, \Phi_{N_\pi}) = L(s, \pi).
\]
\end{thm}
%%%%%
\begin{proof}
By Lemma~\ref{lem:ZL},
we have
$Z(s, W_v, \Phi_{N_\pi})
= L(s, \pi)$ or $L(s, \pi)/L_E(s, \one)$.
Suppose that $Z(s, W_v, \Phi_{N_\pi}) = L(s, \pi)/L_E(s, \one)$.
Take an additive character $\psi_F$ of $F$
whose conductor is $\ri_F$.
Then, by \cite{M2} Proposition 2.8,
we get
\[
Z(1-s, W_v, \hat{\Phi}_{N_\pi})
= q^{-2N_{\pi}(s-1/2)}
Z(1-s, W_v, {\Phi_{N_\pi}}),
\]
and hence
\[
Z(1-s, W_v, \hat{\Phi}_{N_\pi})
= q^{-2N_{\pi}(s-1/2)} L(1-s, \pi)/L_E(1-s, \one)
\]
by assumption.
Due to (\ref{eq:epsilon}),
we obtain
\[
\frac{Z(1-s, W_v, \hat{\Phi}_{N_\pi})}{L(1-s, \pi)}
= \varepsilon(s, \pi, \psi_F, \psi_E)
\frac{Z(s, W_v, {\Phi}_{N_\pi})}{L(s, \pi)},
\]
so that
\[
q^{-2N_{\pi}(s-1/2)}\frac{1}{L_E(1-s, \one)}
= \varepsilon(s, \pi, \psi_F, \psi_E)
\frac{1}{L_E(s, \one)}.
\]
This implies that 
$\varepsilon(s, \pi, \psi_F, \psi_E)$ is not a monomial
in $q^{-2s}$,
which
contradicts Proposition~\ref{prop:mono}.
Therefore we  conclude that 
$Z(s, W_v, \Phi_{N_\pi}) = L(s, \pi)$,
as required.
\end{proof}

We get the following result on $\varepsilon$-factors:
%%%%%%%%%%%%%%%%%
\begin{thm}\label{thm:main2}
Suppose that
$\psi_E$ and $\psi_F$ have conductors
$\ri_E$ and $\ri_F$ respectively.
For
any irreducible generic representation $\pi$
of $G$,
we have
\begin{eqnarray*}
\varepsilon(s, \pi, \psi_F, \psi_E)
= q^{-2N_\pi(s-1/2)}.
\end{eqnarray*}
\end{thm}
%%%%%%%%%%%%%%%%%%
\begin{proof}
The theorem follows from
Theorem~\ref{thm:main}
and \cite{M2} Theorem 3.3.
\end{proof}

\Section{An estimation of $L$-factors}\label{sec:L}
The remaining of this paper is devoted to the proof of
Lemma~\ref{lem:ZL}.
In this section,
we roughly estimate the $L$-factors of generic representations 
of $G$.
To state our result,
we fix the notation for parabolically induced representations.
For a quasi-character $\mu_1$ of $E^\times$
and a character $\mu_2$ of $E^{1}$,
we define a quasi-character $\mu = \mu_1 \otimes \mu_2$ of $T$ by
\begin{eqnarray*}
\mu\left(
\begin{array}{ccc}
a & & \\
& b & \\
& & \overline{a}^{-1}
\end{array}
\right)
= \mu_1(a)\mu_2(b),\ \mathrm{for}\ a \in E^\times\ \mathrm{and}\ b \in E^{1}.
\end{eqnarray*}
We regard $\mu$ as a quasi-character of $B$ 
which is trivial
on $U$.
Let $\mathrm{Ind}_B^G(\mu)$ denote
the normalized parabolic induction.
Then the space of $\mathrm{Ind}_B^G(\mu)$
is that of locally constant functions $f: G \rightarrow \C$
which satisfy
\[
f(bg) = \delta_B(b)^{1/2}\mu(b)f(g),\ \mathrm{for}\
b \in B,\ g \in G,
\]
where $\delta_B$ is the modulus character of $B$.
Note that 
\begin{eqnarray*}
\delta_B\left(
\begin{array}{ccc}
a & & \\
& b & \\
& & \overline{a}^{-1}
\end{array}
\right)
= |a|_E^2,\ \mathrm{for}\ a \in E^\times\ \mathrm{and}\ b \in E^{1}.
\end{eqnarray*}
The group $G$ acts on the space of $\mathrm{Ind}_B^G(\mu)$ 
by a right translation.

Let $(\pi, V)$ be an irreducible generic representation of 
$G$.
To study the integral $Z(s, W)$ of 
$W \in \mathcal{W}(\pi, \psi_E)$,
we recall from \cite{M} section 4.2
some properties of the restriction of 
Whittaker functions to $T_H$.
Let $W$ be a function in $\mathcal{W}(\pi, \psi_E)$.
Under the identification $T_H \simeq E^\times$,
the restriction 
$W|_{T_H}$ of $W$
to $T_H$
is a locally constant function on $E^\times$,
and
there exists an integer $n$ such that
$\supp W|_{T_H} \subset \mi_E^n$.
We set $V(U) = \langle \pi(u)v-v\, |\, v \in V,\ u \in U\rangle$.
Then 
for any element $v$ in $V(U)$, the function
$W_v|_{T_H}$ lies in $\Sch{E^\times}$.

The next lemma follows along the lines in 
the theory of zeta integrals for $\mathrm{GL}(2)$.
However we give a proof
for the reader's convenience.
In the below,
we denote by $\overline{\mu}_1$
the quasi-character of $E^\times$
defined by
$\overline{\mu}_1(a) = \mu_1(\overline{a})$, $a \in E^\times$.

%%%%%%%%%%%%%%%
\begin{lem}\label{lem:L_es}
Let $\pi$ be an irreducible generic representation
of $G$
and $W$ a function in $\mathcal{W}(\pi, \psi_E)$.

(i) Suppose that $\pi$ is supercuspidal.
Then $Z(s, W)$ lies in $\C[q^{-2s}, q^{2s}]$.

(ii) Suppose that $\pi$ is a proper submodule
of $\mathrm{Ind}_B^G (\mu_1 \otimes \mu_2)$,
for some $\mu_1$ and $\mu_2$.
Then 
$Z(s, W)$ belongs to  $L_E(s, \mu_1)\C[q^{-2s}, q^{2s}]$.

(iii) Suppose that $\pi
=\mathrm{Ind}_B^G (\mu_1 \otimes \mu_2)$,
for some $\mu_1$ and $\mu_2$.
Then the integral
$Z(s, W)$ lies in 
$L_E(s, \mu_1) L_E(s, \overline{\mu}_1^{-1}) \C[q^{-2s}, q^{2s}]$.
\end{lem}
%%%%%%%%%
\begin{proof}
Let  $V_U = V/V(U)$ be the normalized
Jacquet module of $\pi$.
The group $T$ acts on 
 $V_U$ by $\delta_B^{-1/2}\pi$.

(i) If $\pi$ is supercuspidal,
then we have $V_U = \{0\}$.
Since $W$ is associated to an element in $V = V(U)$,
the function
$W|_{T_H}$ lies in $\Sch{E^\times}$,
and hence $Z(s, W)$ belongs to $\C[q^{-2s}, q^{2s}]$.

(ii)
In this case,
$V_U$ is isomorphic to $\mu_1 \otimes \mu_2$
as $T$-module.
Take $v \in V$ such that
 $W =W_v$.
 If $v$ lies in $V(U)$,
 then by the proof of (i),
 $\C[q^{-2s}, q^{2s}]$ contains $Z(s, W)$,
 so does 
$L_E(s, \mu_1)\C[q^{-2s}, q^{2s}]$.
Suppose that $v$ does not belong to $V(U)$.
Since $V_U$ is isomorphic to $\mu_1$
as $T_H$-module,
we have $\delta_B^{-1/2}(t(a))\pi(t(a))v- \mu_1(a)v \in V(U)$
for any $a \in E^\times$.
Set $v' = \delta_B^{-1/2}(t(a))\pi(t(a))v- \mu_1(a)v$.
One can observe that
$Z(s, W_{v'}) = (|a|_E^{-s}-\mu_1(a))Z(s, W_v)$.
So
$(|a|_E^{-s}-\mu_1(a))Z(s, W_v)$
lies in $\C[q^{-2s}, q^{2s}]$
for all $a \in E^\times$.

Suppose that $\mu_1$ is ramified.
Then we can find $a \in \ri_E^\times$ such that
$\mu_1(a) \neq 1$.
Thus, we see that
$(1-\mu_1(a))Z(s, W_v)$ lies in $\C[q^{-2s}, q^{2s}]$.
If $\mu_1$ is unramified,
then 
we have $(q^{2s}-\mu_1(\p)) Z(s, W_v) \in 
\C[q^{-2s}, q^{2s}]$ by putting $a = \p$.
These imply that  $Z(s, W_v)$ lies in $L_E(s, \mu_1)\C[q^{-2s}, q^{2s}]$, as required.

(iii) 
In the case when
$\pi
=\mathrm{Ind}_B^G (\mu_1 \otimes \mu_2)$,
there is a $T$-submodule $V_1$ of $V_U$
such that $V_U/V_1 \simeq \mu_1 \otimes \mu_2$ 
and $V_1 \simeq \overline{\mu}_1^{-1} \otimes \mu_2$.
Then we can easily show the assertion by repeating 
the argument
in the proof of (ii) twice.
\end{proof}

According to the classification of representations of 
$G$,
we obtain the following estimation of $L$-factors:
%%%%%%%%%%%%%%%
\begin{prop}\label{prop:L_es}
Let $\pi$ be an irreducible generic representation
of $G$.

(i) Suppose that $\pi$ is supercuspidal.
Then $L(s, \pi)$ divides 
$L_E(s, \one)$.

(ii) Suppose that $\pi$ is a proper submodule
of $\mathrm{Ind}_B^G (\mu_1 \otimes \mu_2)$,
for some $\mu_1$ and $\mu_2$.
Then 
$L(s, \pi)$ divides 
$L_E(s, \mu_1)L_E(s, \one)$.

(iii) Suppose that $\pi
=\mathrm{Ind}_B^G (\mu_1 \otimes \mu_2)$,
for some $\mu_1$ and $\mu_2$.
Then the $L$-factor
$L(s, \pi)$ of $\pi$ divides 
$L_E(s, \mu_1)L_E(s, \overline{\mu}_1^{-1})L_E(s, \one)$.
\end{prop}
%%%%%%%%%
\begin{proof}
Let $W$ and $\Phi$ be functions in $\mathcal{W}(\pi, \psi_E)$ 
and $\Sch{F^2}$ respectively.
Note that 
$W(h)$ and $f(s, h, \Phi)$ are right smooth functions
on $H$.
So 
the integral $Z(s, W, \Phi)$
can be written as a linear combination of 
$Z(s, W')f(s, e, \Phi')$, where 
$W' \in \mathcal{W}(\pi, \psi_E)$ and $\Phi' \in \Sch{F^2}$.
By the theory of zeta integrals for $\mathrm{GL}(1)$,
we see that $f(s, e, \Phi')$ lies in $
L_E(s, \one)\C[q^{-2s}, q^{2s}]$.
So the assertion follows from Lemma~\ref{lem:L_es}.
\end{proof}

\section{Zeta integrals of newforms}\label{section:zeta}
The proof of Lemma~\ref{lem:ZL}
will be done by comparing 
zeta integrals of newforms with Proposition~\ref{prop:L_es}.
To this end,
we give a formula for zeta integrals of newforms.
Let $(\pi, V)$ be an irreducible generic representation
of $G$.
If $N_\pi$ is zero,
then Gelbart and Piatetski-Shapiro in \cite{GPS} computed zeta integrals of 
newforms by using Casselman-Shalika's formula for 
the spherical Whittaker functions in \cite{CS}.
So we treat only representations with $N_\pi > 0$
in this section.
The key is to consider the spaces $V(N_\pi)$ and
$V(N_\pi+1)$ simultaneously,
which are both one-dimensional.
In subsection~\ref{sec:raising},
we recall the definition of the level raising operator
$\theta': V(N_\pi) \rightarrow V(N_\pi+1)$.
The first eigenvalue $\nu$ is defined in subsection~\ref{sec:nu}
as that of the Hecke operator $T$ on $V(N_\pi+1)$.
The second one $\lambda$ is introduced in 
subsection~\ref{sec:lambda}
as the eigenvalue of the composite map
of $\theta'$ and the level lowering operator $\delta:
V(N_\pi+1) \rightarrow V(N_\pi)$.
In subsection~\ref{sec:nu_lambda},
we describe zeta integrals of newforms explicitly
with $\nu$ and $\lambda$ (Theorem~\ref{thm:zeta1}).

%%%%%
\subsection{The level raising operator $\theta'$}\label{sec:raising}
From now on, 
we assume that the conductor of $\psi_E$
is $\ri_E$.
Let $(\pi, V)$ be an irreducible generic representation
of $G$ whose conductor $N_\pi$ is positive.
We abbreviate $N= N_\pi$.
Let $\theta'$ denote the level raising operator
from $V(N)$ to $V(N+1)$ 
defined in \cite{M} section 3.
By \cite{M} Proposition 3.3,
we have
\begin{eqnarray}\label{eq:theta'}
\theta' v = \pi(\zeta^{-1})v + \sum_{x \in \mi_F^{-1-N}/\mi_F^{-N}}
\pi(u(0, x))v,\ v \in V(N),
\end{eqnarray}
where
\[
\zeta
= \left(
\begin{array}{ccc}
\p & & \\
& 1 & \\
& & \p^{-1}
\end{array}
\right).
\]

We fix a newform $v$ in $V(N)$,
and set
\[
c_i = W_v(\zeta^i),\ d_i = W_{\theta' v}(\zeta^i),
\]
for $i \in \Z$.
%%%%%%%%%%%%%%%%%%%%%%%%%%
\begin{lem}\label{lem:theta_rec}
For $i \in \Z$, we have
$d_i = c_{i-1} + qc_i$.
\end{lem}
%%%%%%%%%%%%%%
\begin{proof}
By (\ref{eq:theta'}), we obtain
\begin{align*}
W_{\theta' v}(\zeta^i) &  =   W_v(\zeta^{i-1}) + \sum_{x \in \mi_F^{-1-N}/\mi_F^{-N}}
W_v(\zeta^i u(0, x)),
\end{align*}
for $i \in \Z$.
Since
$\zeta^i u(0, x) = 
u(0, \p^{2i}x)\zeta^i$ and
$\psi_E(u(0, \p^{2i}x)) = 1$,
we obtain
$W_v(\zeta^i u(0, x))
= W_v(\zeta^i)$,
and hence
\begin{align*}
W_{\theta' v}(\zeta^i)
&  =   W_v(\zeta^{i-1}) + 
qW_v(\zeta^i ).
\end{align*}
This implies the lemma.
\end{proof}

%%%%%
\subsection{The eigenvalue $\nu$}\label{sec:nu}
Let $T$ denote the Hecke operator on $V(N+1)$
defined in \cite{M2} subsection 4.1.
For $w \in V(N+1)$, 
we have
\begin{eqnarray*}
Tw  = \frac{1}{\mathrm{vol}(K_N)}\int_{K_N \zeta K_N} \pi(g) w  dg
 = 
\sum_{k \in K_N / K_N \cap \zeta K_N \zeta^{-1}}\pi(k\zeta)w.
\end{eqnarray*}
It follows from \cite{M3} Corollary 5.2
that
the space $V(N + 1)$ is one-dimensional.
So there exists a complex number $\nu$, which is called {\it the Hecke eigenvalue} of $T$, such that
\[
Tw = \nu w
\]
for all $w \in V(N+1)$.
For $w \in V(N+1)$,
we set
\begin{eqnarray}\label{eq:dash}
w' = \sum_{{y \in \mi_E^{N}/\mi_E^{N+1}}}\sum_{ z \in \mi_F^{N}/\mi_F^{N+1}}
\pi(\hat{u}(y, z)) w.
\end{eqnarray}
For each $i \in \Z$, we put 
\[
d'_i = W_{(\theta' v)'} (\zeta^i).
\]
Then we have the following
%%%%%%%%%%%%%%%
\begin{lem}\label{lem:hecke_rec2}
For $i \geq 0$, we have
$\nu d_i = d'_{i-1} + q^4 d_{i+1}$.
\end{lem}
%%%%%%%%%%%%%%
\begin{proof}
By \cite{M2} Lemma 4.4, we obtain
\begin{align}\label{eq:hecke_op}
\nu\theta'v = 
T\theta' v = \pi(\zeta^{-1}) (\theta'v)' +
\sum_{\substack{a \in \ri_E/\mi_E\\b \in \mi_F^{-1-N}/\mi_F^{1-N}}}
\pi({u}(a, b)\zeta ) \theta' v.
\end{align}
Thus,  we get
\begin{align*}
\nu W_{\theta' v}(\zeta^i) & =  W_{(\theta' v)'}(\zeta^{i-1}) +
\sum_{\substack{a \in \ri_E/\mi_E\\b \in \mi_F^{-1-N}/\mi_F^{1-N}}}
W_{\theta' v}(\zeta^i {u}(a, b)\zeta ),
\end{align*}
for $i \geq 0$.
We have $\zeta^i {u}(a, b)
= u(\p^ia, \p^{2i}b)\zeta^i$
and
$\psi_E(u(\p^ia, \p^{2i}b)) = \psi_E(\p^i a) = 1$
because $a \in \ri_E$ and $\psi_E$ has conductor $\ri_E$.
So we get
$W_{\theta' v}(\zeta^i {u}(a, b)\zeta )
= W_{\theta' v}(\zeta^{i+1})$,
and hence
\begin{align*}
\nu W_{\theta' v}(\zeta^i) 
& = W_{(\theta' v)'}(\zeta^{i-1}) +
q^4
W_{\theta' v}(\zeta^{i+1} ).
\end{align*}
This completes the proof.
\end{proof}

%%%%%
\subsection{The eigenvalue $\lambda$}\label{sec:lambda}
The central character $\omega_\pi$ of $\pi$
is trivial on $Z_{N} = Z\cap K_{N}$.
Since the group $Z_N K_{N+1}$ acts on $V(N+1)$ trivially,
we can define the level lowering operator $\delta:
V(N +1) \rightarrow V(N)$ by
\begin{eqnarray*}
\delta w & = &
\frac{1}{\mathrm{vol}(K_{N}\cap (Z_{N}K_{N+1}))}
\int_{K_{N}}\pi(k)w dk
 =  \sum_{k \in K_{N}/K_{N}\cap (Z_{N}K_{N+1})}\pi(k)w,
\end{eqnarray*}
for $w \in V(N+1)$.
Because $V(N)$ is of dimension one,
there exists a complex number $\lambda$ such that
\[
\lambda v = \delta \theta' v
\]
for all  $v \in V(N)$.
%%%%%%%%%%%%%%%
\begin{lem}\label{lem:level_rec2}
We have
\begin{eqnarray*}
& d'_i + q^2 d_{i+1} = \lambda c_i, \ i \geq 0,\\
& d'_{-1} = 0.
\end{eqnarray*}
\end{lem}
%%%%%%%%%%%%%%
\begin{proof}
Since $N$ is positive
and $\omega_\pi$ is trivial on $Z_N$,
we have $N+1 \geq 2$ and $N+1> n_\pi$.
So we can apply
 \cite{M2} Lemma 4.9,
and get
\begin{eqnarray}\label{eq:del}
\lambda v = \delta \theta' v = 
(\theta' v)' + 
 \sum_{y \in \mi_E^{-1}/\ri_E}
\pi( \zeta {u}(y, 0)) \theta' v.
\end{eqnarray}
Hence we obtain
\begin{align*}
\lambda W_v(\zeta^i) 
& =  W_{(\theta'v)'}(\zeta^i) + 
 \sum_{y \in \mi_E^{-1}/\ri_E}
W_{\theta' v}( \zeta^{i+1} {u}(y,0)), 
\end{align*}
for $i \in \Z$.
Because
$\zeta^{i+1} {u}(y,0) 
= u(\p^{i+1}y, 0)\zeta^{i+1}$
and
$\psi_E (u(\p^{i+1}y, 0))
= \psi_E(\p^{i+1}y)$,
we have
$W_{\theta' v}( \zeta^{i+1} {u}(y,0))
=  \psi_E(\p^{i+1}y)
W_{\theta' v}(\zeta^{i+1})$ 
.
So we get
\begin{align*}
\lambda W_v(\zeta^i) 
&=W_{(\theta'v)'}(\zeta^i) + 
 \sum_{y \in \mi_E^{-1}/\ri_E}
 \psi_E(\p^{i+1}y)
W_{\theta' v}(\zeta^{i+1}).
\end{align*}
If $i \geq 0$,
then we have $\psi_E (\p^{i+1}y) = 1$
because $\p^{i+1}y \in \ri_E$
and $\psi_E$ has conductor $\ri_E$.
So we have
\begin{align*}
\lambda W_v(\zeta^i)  & = 
W_{(\theta' v)'}(\zeta^{i}) +
q^2
W_{\theta' v} (\zeta^{i+1}).
\end{align*}
This implies $\lambda c_i = d'_i +q^2 d_{i+1}$,
for $i \geq 0$.

If $i = -1$,
then we have
$\sum_{y \in \mi_E^{-1}/\ri_E} \psi_E (y)
= 0$,
and hence
$\lambda W_v(\zeta^{-1})   = 
W_{(\theta' v)'}(\zeta^{-1})$.
Due to  \cite{M} Corollary 4.6,
we get $W_v(\zeta^{-1}) = 0$.
So we obtain
$W_{(\theta' v)'}(\zeta^{-1}) = 0$.
This implies $d'_{-1} = 0$.
\end{proof}
%%%%
\subsection{Zeta integrals of newforms in $\nu$ and $\lambda$}\label{sec:nu_lambda}
We get the following recursion formula for $c_i = W_v(\zeta^i)$,
$i \geq 0$.
%%%%%%%%%%%%%%%%%%%%
\begin{lem}\label{lem:rec2}
We have
\begin{eqnarray*}
& (\nu+q^2-\lambda )c_{i}
+q(\nu+q^2-q^3)c_{i+1} = q^5 c_{i+2}, \ i \geq 0,\\
& (\nu-q^3) c_{0} = q^4 c_{1}.
\end{eqnarray*}
\end{lem}
%%%%%%%%%%%%%%%%%
\begin{proof}
The assertion follows from Lemmas~\ref{lem:theta_rec},
\ref{lem:hecke_rec2}
and \ref{lem:level_rec2}.
For the second equation,
we note that $c_{-1} = W_v(\zeta^{-1})$
is equal to zero
because of \cite{M} Corollary 4.6.
\end{proof}

By Lemma~\ref{lem:rec2},
we get the following formula of 
zeta integrals of newforms.
%%%%%%%%%%%%%%%%%%%%%%%
\begin{prop}\label{prop:zeta1}
Let $(\pi, V)$ be an irreducible generic representation of $G$
whose conductor $N_\pi$ is positive.
For any $v \in V(N_\pi)$, we have
\[
Z(s, W_v)
=
\frac{(1-q^{-2s})W_v(e)}{1- (\nu+q^2-q^3)q^{-2}q^{-2s}- (\nu+q^2-\lambda )q^{-1}q^{-4s}}.
\]
\end{prop}
%%%%%%%%%%%%%%%%%%%%%%%
\begin{proof}
For $v \in V(N_\pi)$,
it follows from
\cite{M} Corollary 4.6
that $\supp W_v|_{T_H} \subset \ri_E$.
Since $W_v|_{T_H}$ is $\ri_E^\times$-invariant,
we obtain
\begin{align*}
Z(s, W_v)  = 
\sum_{i = 0}^\infty W_v(\zeta^i) |\p^i|_E^{s-1}
 = 
\sum_{i = 0}^\infty c_i q^{2i(1-s)}.
\end{align*}
Put $\alpha = (\nu+q^2-q^3)q^{-4}$
and
$\beta =  (\nu+q^2-\lambda )q^{-5}$.
Then by Lemma~\ref{lem:rec2},
we have
\[
c_{i+2} = \alpha c_{i+1}+\beta c_{i},\ i \geq 0.
\]
So we obtain
\begin{align*}
Z(s, W_v)  & = 
c_0 + c_1q^{2-2s}+
\sum_{i = 0}^\infty (\alpha c_{i+1}+\beta c_{i}) q^{2(i+2)(1-s)}\\
& = 
c_0 + c_1q^{2-2s}
+\beta q^{4-4s}\sum_{i = 0}^\infty  c_{i} q^{2i(1-s)}
+
\alpha q^{2-2s}
\sum_{i = 0}^\infty c_i q^{2i(1-s)}
-\alpha c_0 q^{2-2s}\\
& =  
c_0 + c_1q^{2-2s}
+\beta q^{4-4s}Z(s, W_v) 
+
\alpha q^{2-2s}
Z(s, W_v) 
-\alpha c_0 q^{2-2s}\\
& = 
c_0 + (c_1-\alpha c_0) q^{2-2s}
+
(\alpha q^{2-2s}+\beta q^{4-4s})
Z(s, W_v).
\end{align*}
Thus we have
\begin{eqnarray*}
Z(s, W_v)
 & =& 
\frac{c_0 + (c_1-\alpha c_0) q^{2-2s}}{1-\alpha q^{2-2s}-\beta q^{4-4s}}\\
& =& 
\frac{c_0 (1-q^{-2s})}{1- (\nu+q^2-q^3)q^{-2}q^{-2s}- (\nu+q^2-\lambda )q^{-1}q^{-4s}}.
\end{eqnarray*}
In the last equality,
we use the equation $c_1-\alpha c_0 = -q^{-2}c_0$
from Lemma~\ref{lem:rec2}.
Now the proof is complete.
\end{proof}

%%%%%%%%%%%%%%%%%%%%%%%
\begin{thm}\label{thm:zeta1}
We assume that $\psi_E$ has conductor $\ri_E$.
Let $(\pi, V)$ be an irreducible generic representation of $G$
whose conductor $N_\pi$ is positive.
For the newform $v$ in $V(N_\pi)$
which satisfies $W_v(e) = 1$,
we have
\[
Z(s, W_v, \Phi_{N_\pi})
=
\frac{1}{1- (\nu+q^2-q^3)q^{-2}q^{-2s}- (\nu+q^2-\lambda )q^{-1}q^{-4s}},
\]
where $\nu$ is the eigenvalue of the Hecke operator $T$ on $V(N_\pi+1)$ and $\lambda$ is that of 
the operator $\delta \theta'$ on $V(N_\pi)$.
\end{thm}
%%%%%%%%%%%%%%%%%%%%%%%
\begin{proof}
The theorem follows from
Propositions~\ref{prop:zeta0} and \ref{prop:zeta1}.
\end{proof}

\section{Proof of Lemma~\ref{lem:ZL}}\label{section:pf}
In this section,
we  prove Lemma~\ref{lem:ZL}.
An irreducible generic representation $\pi$ of $G$
is either supercuspidal or 
a submodule of $\mathrm{Ind}_B^G (\mu_1\otimes \mu_2)$,
for some $\mu_1$ and $\mu_2$.
We distinguish the cases:
%%%
\begin{enumerate}
\item[(I)]
$\pi$ is an unramified principal series representation,
that is,
$\pi = \mathrm{Ind}_B^G (\mu_1\otimes \mu_2)$,
where $\mu_1$ is unramified and $\mu_2$ is trivial
(subsection~\ref{sec:I});
\item[(II)]
$\pi$ is supercuspidal or
a submodule of 
$\mathrm{Ind}_B^G (\mu_1\otimes \mu_2)$,
where $\mu_1$ is ramified
(subsection~\ref{sec:II});
\item[(III)]
$\pi$ is a submodule of 
$\mathrm{Ind}_B^G (\mu_1\otimes \mu_2)$,
where $\mu_1$ is unramified,
but $\pi$ is not an unramified principal series representation
(subsection~\ref{sec:III}).
\end{enumerate}
%%%
\begin{rem}\label{rem:1}
We remark that representations in cases (II) and (III)
have positive conductors.
If $\pi$ is generic and supercuspidal, then 
by \cite{M} Corollary 5.5,
we have $N_\pi \geq 2$.
Conductors of the non-supercuspidal representations
are determined in \cite{M3}.
By the proof of Proposition 5.1 in \cite{M3},
if $\pi$ is non-supercuspidal and generic,
then $N_\pi = 0$ implies that $\pi$ is 
an unramified principal series representation.
In particular, the
representations in case (III)
are just the irreducible generic subrepresentations of 
$\mathrm{Ind}_B^G (\mu_1\otimes \mu_2)$
with positive conductors,
where $\mu_1$ runs over the unramified
quasi-characters of $E^\times$.
\end{rem}

%%%%%
\subsection{Proof of Lemma~\ref{lem:ZL}: Case (I)}\label{sec:I}
Let $\mu_1$ be an unramified quasi-character of $E^\times$
and $\mu_2$ the trivial character of $E^1$.
Suppose that  $\pi = \mathrm{Ind}_B^G (\mu_1\otimes \mu_2)$
is irreducible.
We show that Lemma~\ref{lem:ZL} holds for $\pi$.
In this case,
$\pi$ has a non-zero $K_0$-fixed vector.
This implies $N_\pi = 0$.
Let $V$ denote the space of $\pi$
and let $v$ be the element in $V(0)$ which satisfies
$W_v(e) = 1$.
By \cite{GPS} (4.7),
we obtain 
\begin{align*}
Z(s, W_v, \Phi_0) = L_E(s, \mu_1)L_E(s, \overline{\mu}_1^{-1})L_E(s, \one).
\end{align*}
because $\overline{\mu}_1 = \mu_1$.
Due to Proposition~\ref{prop:L_es} (iii),
we have
\begin{align}\label{eq:unram}
Z(s, W_v, \Phi_0) = L(s, \pi)
=
L_E(s, \mu_1)L_E(s, \overline{\mu}_1^{-1})L_E(s, \one),
\end{align}
which completes the proof of 
Lemma~\ref{lem:ZL} in this case.

%%%%%%%%%%%
\subsection{Proof of Lemma~\ref{lem:ZL}: Case (II)}\label{sec:II}
Suppose that an irreducible generic representation
$(\pi, V)$ of $G$ is supercuspidal or
a submodule of 
$\mathrm{Ind}_B^G (\mu_1\otimes \mu_2)$,
where $\mu_1$ is a ramified quasi-character of $E^\times$.
We show the validity of Lemma~\ref{lem:ZL}
for $\pi$.
In this case,
we have
$L(s, \pi) = 1$ or $L_E(s, \one)$
by Proposition~\ref{prop:L_es}.
Let $v$ be the element in $V(N_\pi)$ which satisfies
$W_v(e) = 1$.
Then it follows from Theorem~\ref{thm:zeta1} that
$Z(s, W_v, \Phi_{N_\pi})$ has the form
$1/P(q^{-2s})$, for some $P(X) \in \C[X]$.
Note that
$Z(s, W_v, \Phi_{N_\pi})/L(s, \pi)$ lies in $\C[q^{-2s}, q^{2s}]$ 
by the definition of $L(s, \pi)$.
So one may observe that
$Z(s, W_v, \Phi_{N_\pi}) = 
L(s, \pi)$ or $L(s, \pi)L_E(s, \one)^{-1}$,
as required.

%%%%%
%%%%%
\subsection{Eigenvalues $\nu$ and $\lambda$}\label{sec:eigen}
To prove Lemma~\ref{lem:ZL}
for representations in case (III),
we need more information on 
the eigenvalues $\nu$ and $\lambda$
defined in section~\ref{section:zeta}.
Suppose that
an irreducible generic representation $(\pi, V)$ of $G$
is a submodule of $\mathrm{Ind}_B^G (\mu_1 \otimes \mu_2)$,
where $\mu_1$ is an unramified quasi-character of $E^\times$.
We assume that $N_\pi$ is positive.

\begin{rem}\label{rem:cent}
We identify the center $Z$ of $G$ with $E^1$.
In the case when $\mu_1$ is unramified,
the representation
$\mathrm{Ind}_B^G (\mu_1 \otimes \mu_2)$
admits the central character $\omega_\pi = \mu_2$,
so does $\pi$.
Since $\pi$ has a non-zero $K_{N_\pi}$-fixed vector,
$\omega_\pi = \mu_2$ is trivial on $Z_{N_\pi}
=E^1 \cap (1+\mi_E)^{N_\pi}$.
\end{rem}

We may regard an element in $V$
as a function in $\mathrm{Ind}_B^G (\mu_1 \otimes \mu_2)$.
It follows from \cite{M3} Corollary 4.3
that
every non-zero element $f$ in $V(N_\pi)$
satisfies $f(e) \neq 0$.
By using this property
of newforms,
we show a relation between $\nu$ and $\lambda$.
We abbreviate $N = N_\pi$.
%%%%%%%%%%
\begin{lem}\label{lem:theta_non}
For $f \in V(N)$, we have
\[
(\theta' f) (e) 
 =  (q^2 \mu_1(\p)^{-1}+q)f(e).
\]
In particular, $(\theta' f)(e) \neq 0$
for all non-zero $f \in V(N)$.
\end{lem}
%%%%%%%%%%
\begin{proof}
By (\ref{eq:theta'}),
we have
\[
(\theta' f) (e) =f(\zeta^{-1}) + \sum_{x \in \mi_F^{-1-N}/\mi_F^{-N}}
f(u(0, x)).
\]
Since 
$f$ is 
a function in $\mathrm{Ind}_B^G (\mu_1 \otimes \mu_2)$,
we obtain
$f(\zeta^{-1})
= \delta_B^{1/2}(\zeta^{-1})\mu_1(\p^{-1})
f(e) = q^2 \mu_1(\p)^{-1} f(e)$
and
$f(u(0, x)) = f(e)$.
So we have
\begin{eqnarray*}
(\theta' f) (e)  
=  q^2 \mu_1(\p)^{-1} f(e) + q f(e)
=  (q^2 \mu_1(\p)^{-1}+q)f(e),
\end{eqnarray*}
as required.
For the second assertion,
it suffices to 
claim that $q^2 \mu_1(\p)^{-1}+q \neq 0$.
Since $\mu_1$ is unramified,
if $q^2 \mu_1(\p)^{-1}+q = 0$,
then we have
$\mu_1|_{F^\times} =\omega_{E/F} |\cdot|_F^{-1}$,
where $\omega_{E/F}$ is the non-trivial character of $F^\times$
which is trivial on $N_{E/F}(E^\times)$.
If this is the case, then
it follows from \cite{Keys} that
$\mathrm{Ind}_B^G (\mu_1 \otimes \mu_2)$
is reducible,
and it 
contains no irreducible generic subrepresentations
(see \cite{M3} Lemma 3.6 for instance).
This contradicts the assumption
that 
$\mathrm{Ind}_B^G (\mu_1 \otimes \mu_2)$
contains $\pi$.
\end{proof}

We obtain 
the following relation between $\nu$ and $\lambda$:
%%%%%%%%%%%%%%%%%%%%%%%%%%%%%%%
\begin{lem}\label{lem:eig2}
We have
$\lambda = 
(\nu +q^2 -q^2 \mu_1(\p)) (1+q^{-1}\mu_1(\p))$.
\end{lem}
%%%%%%%%%%%%%%%%%%%%%%%%%%
\begin{proof}
For $f \in V(N)$,
we put
$(\theta' f)' = \sum_{y \in \mi_E^{N}/\mi_E^{N+1}}
\sum_{z \in \mi_F^N/\mi_F^{N+1}}
\pi(\hat{u}(y, z)) \theta' f$
as in (\ref{eq:dash}).
Then by (\ref{eq:hecke_op}),
we obtain
\[
\nu (\theta' f)(e)  
 =  (\theta' f)'(\zeta^{-1})  +
\sum_{\substack{a \in \ri_E/\mi_E\\b \in \mi_F^{-1-N}/\mi_F^{1-N}}}
(\theta'f)({u}(a, b)\zeta ).
\]
Since we regard $\theta' f$ and $(\theta' f)'$ as functions
in $\mathrm{Ind}_B^G (\mu_1 \otimes \mu_2)$,
we have
\[
(\theta' f)'(\zeta^{-1}) =  |\p|_E^{-1}\mu_1(\p^{-1})(\theta' f)'(e)
=q^2\mu_1(\p^{-1})(\theta' f)'(e)
\]
and
\[
(\theta'f)({u}(a, b)\zeta )
= |\p|_E \mu_1(\p)(\theta' f)(e)
= q^{-2} \mu_1(\p)
(\theta'f)(e).
\]
So we get
\begin{eqnarray}\label{eq:551}
\nu (\theta' f)(e)
 =  q^{2}\mu_1(\p^{-1})(\theta' f)'(e)  +
q^2 \mu_1(\p)
(\theta'f)(e). 
\end{eqnarray}

On the other hand,
by (\ref{eq:del}),
we obtain
\[
\lambda f(e)  = 
 (\theta' f)'(e) + 
 \sum_{y \in \mi_E^{-1}/\ri_E}
(\theta' f)( \zeta {u}(y, 0)),
\]
and get
\begin{eqnarray}\label{eq:552}
\lambda f(e)  
 =  (\theta' f)'(e) + 
\mu_1(\p)(\theta' f)(e)
\end{eqnarray}
in a similar fashion.
By (\ref{eq:551}) and (\ref{eq:552}),
we have
\[
 \nu (\theta' f)(e) = 
q^{2}\mu_1(\p^{-1})  (\lambda f(e)-\mu_1(\p)(\theta' f)(e)) +
q^2 \mu_1(\p)
(\theta'f)(e).
\]
According to Lemma~\ref{lem:theta_non},
we obtain
\begin{eqnarray*}
& (\nu +q^2 -q^2 \mu_1(\p)) (q^2 \mu_1(\p)^{-1}+q)f(e) = q^{2}\mu_1(\p^{-1})  \lambda f(e).
\end{eqnarray*}
If $f \in V(N)$ is not zero,
then we get $f(e) \neq 0$.
So this completes the proof.
\end{proof}

By Lemma~\ref{lem:eig2},
we get a formula for zeta integrals of newforms 
with only $\nu$.
%%%%%%%%%%%%%%%%%%%%%%%
\begin{prop}\label{prop:zeta_nu}
We fix a non-trivial additive character $\psi_E$ of $E$
whose conductor is $\ri_E$.
Let $(\pi, V)$ be an irreducible generic representation
of $G$ whose conductor $N_\pi$ is positive
and $v$ the newform for $\pi$
such that $W_v(e) = 1$.
Suppose that
$\pi$ is a subrepresentation of 
$\mathrm{Ind}_B^G (\mu_1 \otimes \mu_2)$,
where $\mu_1$
is an unramified quasi-character of $E^\times$.
Then 
we have
\[
Z(s, W_v, \Phi_{N_\pi})
=
{L_E(s, \mu_1)}
\frac{1}
{1-(\nu+q^2-q^3-q^2\mu_1(\p))q^{-2}q^{-2s}}.
\]
\end{prop}
%%%%%%%%%%%%%%%%%%%
\begin{proof}
By Lemma~\ref{lem:eig2},
we get
\[
\lambda -\nu-q^2
=
(\nu+q^2-q^3-q^2\mu_1(\p))q^{-1}\mu_1(\p),
\]
and hence
\begin{align*}
& 1- (\nu+q^2-q^3)q^{-2}q^{-2s}- (\nu+q^2-\lambda )q^{-1}q^{-4s}\\
= & (1-(\nu+q^2-q^3-q^2\mu_1(\p))q^{-2}q^{-2s})
(1-\mu_1(\p)q^{-2s}).
\end{align*}
So the assertion follows from Theorem~\ref{thm:zeta1}.
\end{proof}

We shall describe the Hecke eigenvalue $\nu$
by values of a function $f$ in $V(N_\pi)$.
Recall that $\nu$ is the eigenvalue of 
the Hecke operator $T$ on $V(N_{\pi}+1)$.
For any integer $i$,
we set
\[
\gamma_i = \hat{u}(\p^{i}, 0)
=
\left(
\begin{array}{ccc}
1 & & \\
\p^i & 1 & \\
-\p^{2i}/2 & -\p^i & 1
\end{array}
\right)\ \mbox{and}\
t_i
=\left(
\begin{array}{ccc}
& & \p^{-i}\\
& 1 & \\
\p^i & & 
\end{array}
\right).
\]
If $n \geq 0$,
then $t_n$ lies in $K_n$.
The following lemma describes
$\nu$
by the values of a function $g$ in $V(N_{\pi}+1)$
at $e$ and $\gamma_{N_\pi}$.
%%%%%%%%%%%%%%%%%%%%%%
\begin{lem}\label{lem:n+1}
For $g \in V(N_\pi +1)$,
we have
\[
\nu g(e) = (q^2(\mu_1(\p) + \mu_1(\p)^{-1})+q^3-q^2)g(e) +q^2(q^2-1)\mu_1(\p)^{-1}g(\gamma),
\]
where $\gamma = \gamma_{N_\pi}$.
\end{lem}
%%%%%%%%%%%%%%%%%%%%%%
\begin{proof}
We abbreviate $N = N_\pi$.
By \cite{M2} Lemma 4.4, we obtain
\[
 \nu g = 
T g = \pi(\zeta^{-1})\sum_{\substack{y \in \mi_E^{N}/\mi_E^{N+1}\\z \in \mi_F^{N}/\mi_F^{N+1}}}
\pi(\hat{u}(y, z)) g +
\sum_{\substack{a \in \ri_E/\mi_E\\b \in \mi_F^{-1-N}/\mi_F^{1-N}}}
\pi({u}(a, b)\zeta ) g.
\]
So we get
\[
 \nu g(e) 
= \sum_{\substack{y \in \mi_E^{N}/\mi_E^{N+1}\\z \in \mi_F^{N}/\mi_F^{N+1}}}
g(\zeta^{-1}\hat{u}(y, z))  +
\sum_{\substack{a \in \ri_E/\mi_E\\b \in \mi_F^{-1-N}/\mi_F^{1-N}}}
g({u}(a, b)\zeta ).
\]
Since we regard  $g$ as an element in 
$\mathrm{Ind}_B^G (\mu_1 \otimes \mu_2)$,
we have
\[
g(\zeta^{-1}\hat{u}(y, z)) 
= |\p|_E^{-1}\mu_1(\p)^{-1} g(\hat{u}(y, z))
= q^2\mu_1(\p)^{-1} g(\hat{u}(y, z))
\]
and
\[
g({u}(a, b)\zeta )
= g(\zeta )
= |\p|_E \mu_1(\p)g(e)
=q^{-2}\mu_1(\p)g(e).
\]
Thus, we get
\begin{align}\label{eq:g}
\nu g(e) =
q^2 \mu_1(\p)^{-1}\sum_{\substack{y \in \mi_E^{N}/\mi_E^{N+1}\\z \in \mi_F^{N}/\mi_F^{N+1}}}
g(\hat{u}(y, z)) +
q^2 \mu_1(\p)g(e).
\end{align}
We shall compute
$g(\hat{u}(y, z)) $, for 
each $y \in \mi_E^{N}/\mi_E^{N+1}$
and $z \in \mi_F^{N}/\mi_F^{N+1}$.

(i)
If 
$y \in \mi_E^{N+1}$
and $z \in \mi_F^{N+1}$,
then 
$\hat{u}(y, z)$ lies in $K_{N+1}$.
Since $g$ is right-invariant under $K_{N+1}$,
we obtain
$g(\hat{u}(y, z)) =g(e)$.

(ii) Suppose that 
$y \not\in \mi_E^{N+1}$ and $z \in \mi_F^{N+1}$.
Then we get $\hat{u}(y, z) 
= \hat{u}(y, 0)\hat{u}(0, z)\equiv \hat{u}(y, 0) \pmod{K_{N+1}}$.
There exists $a \in \ri_E^\times$ such that
$t(a)\hat{u}(y, 0)t(a)^{-1}
= \hat{u}(\p^N, 0) = \gamma$.
Since $g$ is fixed by $K_{N+1}$,
we have
\[
g(\hat{u}(y, z)) = g(\hat{u}(y, 0)) =
g(t(a)^{-1}\gamma t(a))
=g(t(a)^{-1}\gamma).
\]
Because we assume that
$\mu_1$ is unramified,
we get
$g(\hat{u}(y, z)) 
= \mu_1(a^{-1})g(\gamma)
=g(\gamma)$.

(iii) If 
$z \not\in \mi_F^{N+1}$,
then the element
$x = z\e-y\overline{y}/2$ lies in $\mi_E^N\backslash
\mi_E^{N+1}$.
Using the notation in subsection~\ref{subsec:notation},
we write $\hat{u}(y, z) = \hat{\ru}(y, x)$.
Then we have
\[
\hat{\ru}(y, x) = \ru(-\overline{y}/\overline{x}, 1/x) 
\mathrm{diag}(\p^{N+1}/\overline{x}, -\overline{x}/x, \p^{-1-N}x)t_{N+1}
\ru(-\overline{y}/x, 1/{x}).
\]
One can observe that 
$t_{N+1}
\ru(-\overline{y}/x, 1/{x})$ lies in $K_{N+1}$.
Since 
$g$ is an element in $\mathrm{Ind}_B^G (\mu_1 \otimes \mu_2)$ fixed by $K_{N+1}$,
we have
\[
g(\hat{u}(y, z)) = 
g(\hat{\ru}(y, x)) =
g(\mathrm{diag}(\p^{N+1}/\overline{x}, -\overline{x}/x, \p^{-1-N}x)).
\]
The assumption
$x \in \mi_E^N\backslash
\mi_E^{N+1}$
implies 
$\p^{N+1}/\overline{x} \in \p \ri_E^\times$,
so 
we get
$g(\hat{u}(y, z)) = q^{-2}\mu_1(\p)\mu_2(-\overline{x}/x)
g(e)$.
Note that
$x+\overline{x}+y\overline{y} = 0$,
and hence
$-\overline{x}/x = 1+y\overline{y}/x $.
Since $y \in \mi_E^N$ and $x \in \mi_E^N\backslash
\mi_E^{N+1}$,
we obtain
$-\overline{x}/x \in 1+\mi_E^N$.
Thus,
by Remark~\ref{rem:cent},
we see that $\mu_2(-\overline{x}/x ) = 1$,
so that
$g(\hat{u}(y, z)) = q^{-2}\mu_1(\p)g(e)$.

\medskip
By (\ref{eq:g}) and the above consideration,
we conclude that
\begin{align*}
\nu g(e)& = 
q^2\mu_1(\p)^{-1} (g(e) + (q^2-1)g(\gamma) +q^{-2}\mu_1(\p)q^2(q-1)g(e))+q^2 \mu_1(\p)g(e)\\
& = 
(q^2(\mu_1(\p) + \mu_1(\p)^{-1})+q^3-q^2)g(e)
+q^2(q^2-1)\mu_1(\p)^{-1}g(\gamma).
\end{align*}
This completes the proof.
\end{proof}

Applying Lemma~\ref{lem:n+1}
to $g = \theta' f$, where $f \in V(N_\pi)$,
we get the following
%%%%%%%%%%%%%%%%%%%%%%
\begin{lem}\label{lem:gamma}
For any non-zero element $f$ in $V(N_\pi)$,
we have
\[
\nu = q^2(\mu_1(\p) + \mu_1(\p)^{-1})+q^3-q^2+q^2(q^2-1)\mu_1(\p)^{-1} (q^2 \mu_1(\p)^{-1}+q)^{-1}(\theta' f)(\gamma)/f(e),
\]
where $\gamma = \gamma_{N_\pi}$.
\end{lem}
%%%%%%%%%%%%%%%%%%%%%%%
\begin{proof}
Put $g = \theta' f \in V(N_\pi+1)$.
By Lemma~\ref{lem:theta_non},
we have
$g(e) 
 =  (q^2 \mu_1(\p)^{-1}+q)f(e) \neq 0$.
 So the assertion follows from Lemma~\ref{lem:n+1}.
\end{proof}

We apply Lemma~\ref{lem:gamma}
to zeta integrals of newforms.
%%%%%%%%%%%%%%%%%%%%%%%
\begin{prop}\label{prop:zeta3}
Under the same assumption of Proposition~\ref{prop:zeta_nu},
we have
\[
Z(s, W_v, \Phi_{N_\pi})
=
{L_E(s, \mu_1)}
\frac{1}
{1-\alpha q^{-2s}}.
\]
Here $\alpha$ is given by
\[
\alpha = \mu_1(\p)^{-1}+\mu_1(\p)^{-1}
(q^2-1)(q^2 \mu_1(\p)^{-1}+q)^{-1}(\theta' f)(\gamma_{N_\pi})/f(e),
\]
for any non-zero function $f$ in $V(N_\pi)$.
\end{prop}
%%%%%
\begin{proof}
The proposition follows from Proposition~\ref{prop:zeta_nu}
and
Lemma~\ref{lem:gamma}.
\end{proof}

%%%%%%%%%%%%%%%%%%%%%%%%%%%%%%%%%
%%%%%%%%%%%%%%%%%%%%%%%%%%%%%%%%%
\subsection{Proof of Lemma~\ref{lem:ZL}: Case (III)}\label{sec:III}
We shall finish the proof of Lemma~\ref{lem:ZL}.
The remaining representations are those in case (III).
Let $(\pi, V)$ be an irreducible generic representation
of $G$ whose conductor is positive.
We suppose that 
$\pi$ is a subrepresentation of 
$\mathrm{Ind}_B^G (\mu_1 \otimes \mu_2)$,
where $\mu_1$ is unramified.

Firstly, we assume that 
$\pi$ is a proper submodule of 
$\mathrm{Ind}_B^G (\mu_1 \otimes \mu_2)$.
Then Proposition~\ref{prop:L_es}
implies
that $L(s, \pi) = L_E(s, \mu_1)$
or $L_E(s, \mu_1)L_E(s, \one)$.
Let $v$ be the newform in $V(N_\pi)$
such that $W_v(e)= 1$.
It follows from
Proposition~\ref{prop:zeta3}
that
$Z(s, W_v, \Phi_{N_\pi})$
has the form
$L(s, \mu_1) \cdot( 1/P(q^{-2s}))$,
for some $P(X) \in \C[X]$.
Because $Z(s, W_v, \Phi_{N_\pi})/L(s, \pi)$
lies in $\C[q^{-2s}, q^{2s}]$,
we must have
$Z(s, W_v, \Phi_{N_\pi}) = L(s, \pi)$
or $L(s, \pi)/L_E(s, \one)$.

Secondly, we consider the case when
$\pi
=\mathrm{Ind}_B^G (\mu_1 \otimes \mu_2)$.
The assumption $N_\pi > 0$
implies that $\mu_2$ is not trivial.
In this case,
we can show 
Lemma~\ref{lem:ZL}
by comparing 
Proposition~\ref{prop:L_es}
with the following one in a similar fashion:
%%%%%%%%%%%%%%%%%%%%%%
\begin{prop}\label{prop:eigen_1}
Let
$\mu_1$ be an unramified quasi-character of $E^\times$
and $\mu_2$ a non-trivial character of $E^1$.
Suppose that $\pi = \mathrm{Ind}_{B}^G (\mu_1 \otimes \mu_2)$
is irreducible.
Then
we have
\[
Z(s, W_v, \Phi_{N_\pi}) 
=
L_E(s, \mu_1)L_E(s, \overline{\mu}_1^{-1}),
\]
where $v$ is the newform in $V(N_\pi)$ such that
$W_v(e)= 1$.
\end{prop}
%%%%%%%%%%%%%%%%%%%%%%
\begin{proof}
Set $\gamma = \gamma_{N_\pi}$.
Since $\mu_1$ is unramified,
we have $\overline{\mu}_1 = \mu_1$.
By Proposition~\ref{prop:zeta3},
it enough to show that 
$\theta' f(\gamma) = 0$,
for any functions $f$ in $V(N_\pi)$.
By \cite{M3} Theorem 2.4 (ii),
the space of $K_{N_\pi+1}$-fixed vectors in 
$\mathrm{Ind}_B^G (\mu_1 \otimes \mu_2)$
is one-dimensional and consists of the functions
whose supports are contained in $BK_{N_\pi+1}$
since we assume that $\mu_1$ is unramified.
Due to \cite{M3} Lemma 2.1,
the sets
$B\gamma K_{N_\pi+1}$ and 
$BK_{N_\pi+1} = B\gamma_{N_\pi+1} K_{N_\pi+1}$
are disjoint.
So for any $f \in V(N_\pi)$,
we get $(\theta' f)(\gamma) = 0$
because $\theta' f$ is fixed by $K_{N_\pi+1}$.
This completes the proof.
\end{proof}

Now the proof of Lemma~\ref{lem:ZL}
is complete.

\section{An example of a computation of $L$-factors}\label{sec:example}
Let $(\pi, V)$ be an irreducible generic representation of $G$
whose conductor $N_\pi$ is positive.
Suppose that $\pi$ is a subrepresentation of 
$\mathrm{Ind}_B^G (\mu_1 \otimes \mu_2)$,
where $\mu_1$ is an unramified quasi-character of $E^\times$
and $\mu_2$ is a character of $E^1$.
In this section,
we determine the $L$-factor of $\pi$
by using
the results in subsection~\ref{sec:eigen}.
%%%
\subsection{Irreducible case}
Suppose that
$\mathrm{Ind}_B^G (\mu_1 \otimes \mu_2)$
is irreducible.
Then we have
$\pi = \mathrm{Ind}_B^G (\mu_1 \otimes \mu_2)$
and $\mu_2$ is not trivial
because we assume that $N_\pi > 0$.
%%%
\begin{prop}\label{prop:iu}
Let $\mu_1$ be an unramified quasi-character of $E^\times$
and $\mu_2$ a non-trivial character of $E^1$.
Suppose that 
$\pi = \mathrm{Ind}_B^G (\mu_1 \otimes \mu_2)$
is irreducible.
Then we have
\[
L(s, \pi) = L_E(s, \mu_1)L_E(s, \overline{\mu}_1^{-1}).
\]
\end{prop}
%%%
\begin{proof}
Theorem~\ref{thm:main} and Proposition~\ref{prop:eigen_1}
imply the assertion.
\end{proof}

%%%
%%%
\subsection{Reducible case}
Suppose that
$\mathrm{Ind}_B^G (\mu_1 \otimes \mu_2)$
is reducible.
Recall that we assume that 
$\mathrm{Ind}_B^G (\mu_1 \otimes \mu_2)$
contains an irreducible generic subrepresentation $\pi$.
So, by \cite{Keys},
there are the following three cases:
\begin{enumerate}
\item[(RU1)]
$\mu_1 = |\cdot|_E$ and $\mu_2$ is trivial:
Then $\pi$ is the Steinberg representation $\mathrm{St}_G$
of $G$
and $N_\pi = 2$ by \cite{M3} Proposition 3.4 (i).
(Proposition~\ref{prop:RU1})

\item[(RU2)]
$\mu_1|_{F^\times} = \omega_{E/F} |\cdot|_F$,
where $\omega_{E/F}$ denotes the 
non-trivial character of $F^\times$
which is trivial on $N_{E/F}(E^\times)$.
By \cite{M3} Proposition 3.7,
we have
$N_\pi = c(\mu_2)+1$.
(Propositions~\ref{prop:RU2} and \ref{prop:RU1})

\item[(RU3)]
$\mu_1$ is trivial and $\mu_2$ is not trivial:
Then due to \cite{M3} Proposition 3.8,
we get
$N_\pi = c(\mu_2)$.
(Proposition~\ref{prop:RU3})
\end{enumerate}
Here $c(\mu_2)$ denotes the conductor of $\mu_2$,
that is,
\[
c(\mu_2)
= \min\{n \geq 0\, |\, \mu_2|_{E^1\cap (1+\mi_E)^n} = 1\}.
\]
We fix a non-trivial additive character $\psi_E$ of 
$E$ with conductor $\ri_E$.
Let $v$ be the newform for $\pi$
such that $W_v(e)= 1$.
Then by Theorem~\ref{thm:main},
we have
$Z(s, W_v, \Phi_{N_\pi}) = L(s, \pi)$.
We regard elements in $V$
as functions in $\mathrm{Ind}_B^G (\mu_1 \otimes \mu_2)$.
By Proposition~\ref{prop:zeta3},
to determine  $L(s, \pi) = Z(s, W_v, \Phi_{N_\pi})$,
it is enough to compute
$(\theta' f)(\gamma_{N_\pi})/f(e)$,
where $f$ is a non-zero function in $V(N_\pi)$.
We shall determine 
$(\theta' f)(\gamma_{N_\pi})/f(e)$
explicitly, for each case.

\subsection{Case (RU3)}
We  consider the case (RU3).
%%%
\begin{prop}\label{prop:RU3}
Let
$\mu_2$ be a non-trivial character of $E^1$
and $(\pi, V)$ the irreducible generic subrepresentation of 
$\mathrm{Ind}_{B}^G (\one \otimes \mu_2)$.
Then
we have
\[
L(s, \pi)
=
L_E(s, \one)^2.
\]
\end{prop}
%%%
\begin{proof}
It follows from \cite{M3} Proposition 3.8
that
$V(n)$ coincides with 
the space of  $K_n$-fixed vectors in
$\mathrm{Ind}_{B}^G (\one \otimes \mu_2)$
for all $n$.
So we may apply the argument in the proof of 
Proposition~\ref{prop:eigen_1},
and get 
$(\theta' f)(\gamma_{N_\pi}) = 0$,
for any $f \in V(N_\pi)$.
By Proposition~\ref{prop:zeta3},
we obtain
$Z(s, W_v, \Phi_{N_\pi}) 
=L_E(s, \one)^2$,
where  $v$ is the newform in $V(N_\pi)$ such that
$W_v(e)= 1$.
The assertion follows from this and Theorem~\ref{thm:main}.
\end{proof}

%%%
%%%
\subsection{Case (RU2-I)}
Let us consider the case (RU2).
We further assume that $\mu_2$
is trivial.
The remaining case is treated in the next subsection.
Then $\mathrm{Ind}_B^G (\mu_1 \otimes \mu_2)$
has the trivial central character,
so does $\pi$.
By \cite{M3} Proposition 3.7,
we get
$N_\pi = 1$.
Since
$\mu_1|_{F^\times} = \omega_{E/F} |\cdot|_F$,
we have $\mu_1(\p) = -q^{-1}$.

%%%%%%%%%%%%%%%%%
\begin{lem}\label{lem:RU21}
For $f \in V(1)$,
we have
\[
(\theta' f)(\gamma_1) = (q+1)f(e).
\]
\end{lem}
%%%%%%%%%%%%%%%%%
\begin{proof}
We abbreviate $\gamma = \gamma_1$.
Set $g = \theta' f \in V(2)$
and $\gamma' =
t_2 \gamma t_2 = u(-\p^{-1}, 0)$.
We have
$\gamma = t_2 \gamma' t_2
= \zeta^{-1} t_1 \gamma' t_2$.
Since $g$ is a function in $\mathrm{Ind}_B^G \mu_1 \otimes \mu_2$ which is fixed by $K_2$
and $t_2 \in K_2$,
we obtain
$g(\gamma) 
= g(\zeta^{-1}t_1 \gamma' t_2)
= q^2 \mu_1(\p^{-1}) g(t_1 \gamma')$.
By 
(\ref{eq:theta'}),
we get
\begin{align*}
 g(t_1 \gamma') = f(t_1 \gamma' \zeta^{-1}) + \sum_{x \in \mi_F^{-2}/\mi_F^{-1}}
f(t_1 \gamma' u(0, x)),
\end{align*}
and hence
\begin{align}\label{eq:8}
 g(\gamma) =q^2 \mu_1(\p^{-1})
f(t_1 \gamma' \zeta^{-1}) + q^2 \mu_1(\p^{-1})\sum_{x \in \mi_F^{-2}/\mi_F^{-1}}
f(t_1 \gamma' u(0, x)).
\end{align}

Firstly, we have
$t_1 \gamma' \zeta^{-1} = t_1 \zeta^{-1}\zeta \gamma' \zeta^{-1}$.
Note that $t_1 \zeta^{-1} = \zeta t_1$
and $\zeta \gamma' \zeta^{-1}
= u(-1, 0)$.
We get $t_1 \gamma' \zeta^{-1} =\zeta t_1u(-1, 0)$.
Since $t_1u(-1, 0) \in K_1$ and $f \in V(1)$,
we obtain
\begin{eqnarray*}
f(t_1 \gamma' \zeta^{-1}) 
 =  f(\zeta t_1u(-1, 0))
 = f(\zeta)
 =  q^{-2}\mu_1(\p) f(e).
\end{eqnarray*}
Secondly,
we get
$t_1 \gamma' u(0, x)
= t_1 u(-\p^{-1}, x)
= \hat{u}(1, \p^2 x) t_1$.
Since $t_1 \in K_1$ and $f \in V(1)$,
we obtain
\begin{eqnarray*}
f(t_1 \gamma' u(0, x)) = 
f(\hat{u}(1, \p^2 x) t_1)
=f(\hat{u}(1, \p^2 x)).
\end{eqnarray*}
Set $z = \p^2 x\e-1/2$.
Then $z$ lies in $\ri_E^\times$
because $\p^2 x \in \mi_E^2$.
With the notation in subsection~\ref{subsec:notation},
we write $\hat{u}(1, \p^2 x)
= \hat{\ru}(1, z)$.
We use the relation
\[
\hat{\ru}(1, z) = \ru(-1/\overline{z}, 1/z) \mathrm{diag}(\p/\overline{z}, -\overline{z}/z, \p^{-1}z)t_{1}\ru(-{1}/z, 1/z).
\]
By $z \in \ri_E^\times$,
we have $t_{1}\ru(-{1}/z, 1/z) \in K_1$.
Recall that $f$ is a function in $(\mathrm{Ind}_B^G \mu_1 \otimes \mu_2)$ which is fixed by $K_1$.
So we obtain
\[
f(t_1 \gamma' u(0, x))
=
f(\mathrm{diag}(\p/\overline{z}, -\overline{z}/z, \p^{-1}z))
= q^{-2}\mu_1(\p)f(e)
\]
because $z$ lies in $\ri_E^\times$ and
we assume that
$\mu_2$ is trivial.
Finally, by (\ref{eq:8}),
we get
$g(\gamma) = (q+1)f(e)$,
as required.
\end{proof}
%%%%%%%%%%%%%%%%%%%%%%
\begin{prop}\label{prop:RU2}
Let $\mu_1$ be an unramified quasi-character of $E^\times$
which satisfies
$\mu_1|_{F^\times} = \omega_{E/F} |\cdot|_F$,
and $\mu_2$ the trivial character of $E^1$.
For the irreducible generic subrepresentation 
$\pi$ of $\mathrm{Ind}_B^G (\mu_1\otimes \mu_2)$,
we have
\[
L(s, \pi)
=\displaystyle
L_E(s, \mu_1)L_E(s, \one).
\]
\end{prop}
%%%%%%%%%%%%%%%%%%%%
\begin{proof}
We may apply Proposition~\ref{prop:zeta3}.
Due to Lemma~\ref{lem:RU21},
the number $\alpha$ in Proposition~\ref{prop:zeta3}
satisfies
\begin{eqnarray*}
\alpha = 
\mu_1(\p)^{-1}+\mu_1(\p)^{-1}
(q^2-1)(q^2 \mu_1(\p)^{-1}+q)^{-1}(q+1) =  1,
\end{eqnarray*}
since $\mu_1(\p) = -q^{-1}$.
Now the assertion follows from Theorem~\ref{thm:main}
and Proposition~\ref{prop:zeta3}.
\end{proof}

%%%%%%%%%%%
%%%%%%%%%%%
\subsection{Cases (RU1) and (RU2-II)}
Suppose that an irreducible generic representation $\pi$ 
of $G$ is a subrepresentation of 
$\mathrm{Ind}_G (\mu_1 \otimes \mu_2)$.
We assume that $\mu_1$ and $\mu_2$
satisfy one of the following conditions:
\begin{enumerate}
\item
$\mu_1 = |\cdot|_E$ and $\mu_2$ is trivial;
\item
$\mu_1$ is an unramified quasi-character of $E^\times$
such that
$\mu_1|_{F^\times} = \omega_{E/F} |\cdot|_F$,
and $\mu_2$ is a non-trivial character of $E^1$.
\end{enumerate}

In the first case,
we have
$N_\pi = 2$ 
by \cite{M3} Proposition 3.4 (i),
and $\pi$ has the trivial central character.
In the second case,
we get
$N_\pi = c(\mu_2) + 1 \geq 2$
by \cite{M3} Proposition 3.7,
and $n_\pi = c(\mu_2)$
by Remark~\ref{rem:cent}.

%%%%%%%%%%%%%%%%%%%%%%
\begin{prop}\label{prop:RU1}
Suppose that an irreducible generic representation  $\pi$
satisfies one of the assumptions in this subsection.
Then we have
\[
L(s, \pi)
=\displaystyle
{L_E(s, \mu_1)}.
\]
\end{prop}
%%%%%%%%%%%%%
\begin{proof}
In both cases,
we have $N_\pi \geq 2$ and $N_\pi > n_\pi$.
So
we may apply the results in \cite{M2}.
Suppose that $\psi_E$ has conductor $\ri_E$.
Let $v$ be the newform for $\pi$ such that
$W_v(e) = 1$.
Then by Proposition~\ref{prop:zeta0}
and
\cite{M2} Proposition 4.12,
we see that $Z(s, W_v, \Phi_{N_\pi})$
has the form
$1/P(q^{-2s})$,
where $P(X)$
is a polynomial in $\C[X]$
such that
$P(0) = 1$ and $\deg P(X) \leq 1$.
So Proposition~\ref{prop:zeta3}
implies that 
$Z(s, W_v, \Phi_{N_\pi}) = L_E(s, \mu_1)$.
Now the assertion follows from
Theorem~\ref{thm:main}.
\end{proof}

\section*{Acknowledgements}
The author would like to thank Takuya Yamauchi for helpful discussions.

%\bibliographystyle{amsplain}
%\bibliography{bib}

\providecommand{\bysame}{\leavevmode\hbox to3em{\hrulefill}\thinspace}
\providecommand{\MR}{\relax\ifhmode\unskip\space\fi MR }
\providecommand{\MRhref}[2]{%
  \href{http://www.ams.org/mathscinet-getitem?mr=#1}{#2}
}
\providecommand{\href}[2]{#2}

\end{document}